%% file: Kaye-AbielSystem.tex
\documentclass{proc-l}

\usepackage{color}
\usepackage{xcolor}
\usepackage{amsmath}
\usepackage{subfigure}
\usepackage{float}

\usepackage{graphicx}

\newtheorem{thm}{Theorem}[section]
\newtheorem{lem}[thm]{Lemma}
\newtheorem{prop}[thm]{Proposition}
\newtheorem{cor}[thm]{Corollary}

\theoremstyle{definition}

\theoremstyle{remark}
\newtheorem{rem}[thm]{Remark}

\numberwithin{equation}{section}

\begin{document}

\title[Extremal Parameters Curve]{On The Extremal Parameters Curve of a Quasilinear Elliptic System of Differential Equations }


%    author one information


%    author two information
\author{Kaye Silva}
\address{Universidade Federal de Goi\'{a}s}
\curraddr{}
\email{kayeoliveira@hotmail.com}
\thanks{}

\author{Abiel Macedo}
\address{Universidade Federal de Goi\'{a}s}
\curraddr{}
\email{abiel@gmail.com}
\thanks{}

%    \subjclass is required.

\date{}

\dedicatory{}

%    "Communicated by" -- provide editor's name; required.
\commby{}

%    Abstract is required.
\begin{abstract} We consider a system of quasilinear elliptic equations, with indefinite super-linear nonlinearity, depending on two real parameters $\lambda,\mu$. By using the Nehari manifold and the notion of extremal parameter, we extend some results concerning existence of positive solutions.
\end{abstract}

\maketitle

\section{Introduction}

In this work we study the following system of quasilinear elliptic equations

\begin{equation}\label{pq}\tag{$p,q$}
\left\{
\begin{aligned}
-\Delta_p u &= \lambda |u|^{p-2}u+\alpha f|u|^{\alpha-2}|v|^\beta u 
&\mbox{in}\ \ \Omega, \nonumber\\ 
-\Delta_q v &= \mu |v|^{q-2}v+\beta f|u|^{\alpha}|v|^{\beta-2}v
&\mbox{in}\ \ \Omega, \\
&(u,v)\in W_0^{1,p}(\Omega)\times W_0^{1,q}(\Omega). \nonumber
\end{aligned}
\right.
\end{equation}
where $\Omega\subset \mathbb{R}^N$ is a bounded domain with regular boundary, $\lambda,\mu\in\mathbb{R}$, $1<p,q<\infty$ and 

\begin{equation}\label{alphabeta}\tag{$\alpha,\beta$}
\frac{\alpha}{p}+\frac{\beta}{q}>1,\ \ \alpha> p\ \mbox{or}\ \ \beta>q,\ \ \frac{\alpha}{p^*}+\frac{\beta}{q^*}<1.
\end{equation}

The symbols $-\Delta_p,-\Delta_q$ denotes the $p$ and $q$ Laplacian operators and $p^*$ and $q^*$ are the critical Sobolev exponents. We assume that the nonlinearity is indefinite, that is $f\in L^\infty(\Omega)$ and $f^+\equiv \max\{f,0\}$, $f^-\equiv\max\{-f,0\}$ are not identically zero in $\Omega$. From now on $(\lambda_1,\varphi_1)$ and $(\mu_1,\psi_1)$ denotes the first eingenpair of the operators $-\Delta_p$ and $-\Delta_q$ respectively on $\Omega$. We say that $(u,v)\in  W_0^{1,p}(\Omega)\times W_0^{1,q}(\Omega)$ is a positive solution of \eqref{pq} if $u(x)>0$, $v(x)>0$ for all $x\in \Omega$ and  $(u,v)$ is a critical point of the energy functional defined by
\begin{equation}
\Phi_\sigma(u,v)=\frac{1}{p}\left(\|\nabla u\|_p^p-\lambda \|u\|_p^p\right)+\frac{1}{q}\left(\|\nabla v\|_q^q-\lambda \|v\|_q^q\right)+F(u,v),
\end{equation}
where $\sigma=(\lambda,\mu)$, $F(u,v)=\int_{\Omega} f|u|^\alpha |v|^\beta$  and $\|\cdot\|_p$, $\|\cdot\|_q$ are the standard $L^p$ and $L^q$ norm on $\Omega$. We consider $W_0^{1,p}(\Omega)$ and $W_0^{1,q}(\Omega)$ with the standard Sobolev norms $\|u\|_{1,p}=\|\nabla u\|_p$ and $\|v\|_{1,q}=\|\nabla v\|_q$. We intend to understand the set of parameters for which positive solutions of \eqref{pq} does exist. Define $\lambda^*=\sigma^*\lambda_1$ and $\mu^*=\sigma^*\mu_1$ where

\begin{equation*}
\sigma^*=\inf_{(u,v)\in  W_0^{1,p}(\Omega)\times W_0^{1,q}(\Omega)}\left\{\max\left\{\frac{1}{\lambda_1}\frac{\int|\nabla u|^p}{\int |u|^p},\frac{1}{\mu_1}\frac{\int|\nabla v|^q}{\int |v|^q}\right\}:\ F(u,v)\ge 0\right\}.
\end{equation*}
 Following the works of Bozhkov-Mitidieri \cite{bozmit} and Bobkov-Il'yasov \cite{bobil,bobil2} one has the following structure (see Figure \ref{8}):
 
\begin{figure}[H]
	\centering\footnotesize   
	{\def\svgwidth{0.8\linewidth}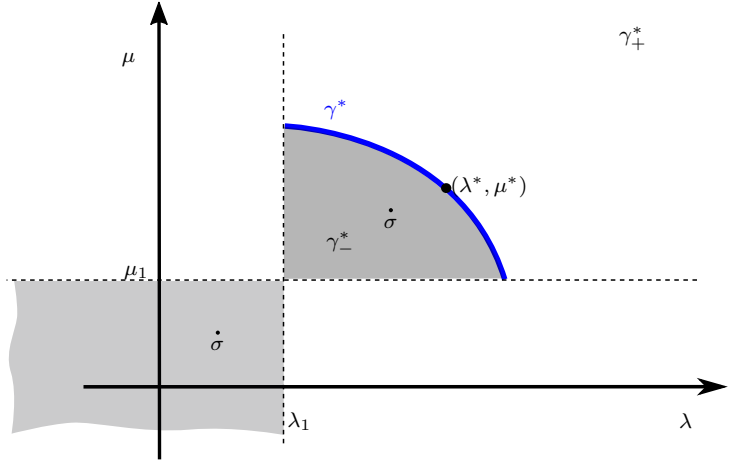}
	\caption{Extremal Parameters Curve}
	\label{8}
\end{figure}
 
 \begin{description}
 	\item[I)] if $\sigma=(\lambda,\mu)\in (-\infty,\lambda_1)\times (-\infty,\mu_1)$, then the problem \eqref{pq} admits at least one positive solution $(u,v)$ with $\Phi_\sigma (u,v)>0$;
 	\item[II)] there exists a curve $\gamma^*$, which determines a region $\gamma^*_-$ over the parameter plane such that for each $(\lambda,\mu)\in \gamma^*_-$, the problem \eqref{pq} admits at least one positive solution $(u,v)$ with $\Phi_\sigma(u,v)<0$.
 \end{description}

 In both works, the authors uses the fibering Method of Pohozaev \cite{poh}. We extend these results by showing the existence of positive solutions when $(\lambda,\mu)\in \gamma^*_+\cup\gamma^*$. The Nehari manifold (see Nehari \cite{neh1}) si defined by
 \begin{equation*}
 \mathcal{N}_\sigma =\{(u,v)\in  W_0^{1,p}(\Omega)\times W_0^{1,q}(\Omega):\ D\Phi_\sigma(u,v)(u,v)=0  \},
 \end{equation*}
 where $D\Phi_\sigma (u,v)$ denotes the Frechet derivative of $\Phi_\sigma$. Let $\Omega^+=\{x\in \Omega:\ f(x)>0\}$, $\Omega^-=\{x\in \Omega:\ f(x)<0\}$ and $\Omega^0=\{x\in \Omega:\ f(x)=0\}$. In this paper we assume one of the following hypothesis

\begin{description}
	\item[$f_1$] the measure of $\Omega^0$ is zero;
	\item[$f_2$] the measure of $\Omega^0$ is not zero and the interior of $\Omega^0\cup \Omega^+$ is a regular domain. Moreover $\operatorname{int}(\Omega^0\cup \Omega^+)$ contains a connected component which intersects both $\Omega^0$ and $\Omega^+$.
\end{description}

Our first result is the following

\begin{thm}\label{AKGLOBAL} Assume that \eqref{alphabeta} is satisfied and $F(\varphi_1,\psi_1)<0$. If $f_1$ or $f_2$ are satisfied, then there exists a curve $\gamma^*\subset \mathbb{R}^2$ which is the union of the graph of two continuous decreasing functions $\mu_{ext}:(\lambda_1,\lambda^*]\to \mathbb{R}$ and $\lambda_{ext}:(\mu_1,\mu^*]\to \mathbb{R}$ satisfying $\mu_{ext}(\lambda^*)=\mu^*$, $\lambda_{ext}(\mu^*)=\lambda ^*$ and
	
	\begin{description}
			\item[i)] for each $\sigma \in \gamma^*$, there exists a positive solution to \eqref{pq} with zero energy;
		\item[ii)] for each $\sigma=(\lambda,\mu)$ with $(\lambda\in (\lambda_1,\lambda^*] \mbox{ and } \mu\le\mu_{ext}(\lambda))$ or $(\mu\in (\mu_1,\mu^*] \mbox{ and } \lambda\le\lambda_{ext}(\mu))$, the energy functional $\Phi_\sigma$ is bounded from below over the Nehari manifold $\mathcal{N}_\sigma$;
		\item[iii)]for each $\sigma=(\lambda,\mu)$ with $(\lambda\in (\lambda_1,\lambda^*] \mbox{ and } \mu_{ext}(\lambda)<\mu)$ or $(\mu\in (\mu_1,\mu^*] \mbox{ and } \lambda_{ext}(\mu)<\lambda)$ or $(\lambda>\lambda^* \mbox{ and } \mu>\mu*)$, the energy functional $\Phi_\sigma$ is unbounded from below over the Nehari manifold $\mathcal{N}_\sigma$;

	\end{description}
	
\end{thm}

Problems with indefinite nonlinearity has a long history (see for example Alama-Tarantello \cite{alam}, Ouyang \cite{ou2}, Berestycki-Capuzzo Dolcetta-Nirenberg \cite{becapunire}, Il'yasov \cite{ilyas} and the references therein). In general, a problem with indefinite nonlinearity possess multiplicity of positive solutions. This multiplicity is related with the behavior of the Nehari manifold accordingly with the parameter. In Il'yasov \cite{ilyasENMM}, the author formalized the notion of extremal parameter. In terms of the Nehari manifold, the extremal parameter $\sigma^*$ defines a threshold for which minimization of the energy functional $\Phi_\sigma$ over $\mathcal{N}_\sigma$ is available. In fact, the curve $\gamma^*$, which from now on we will call extremal parameters curve, divides the parameter space $(\lambda,\mu)$ in such a way that above the curve the energy of $\Phi_\sigma$ over $\mathcal{N}_\sigma$ is unbounded from below while under the curve the energy of $\Phi_\sigma$ over $\mathcal{N}_\sigma$ is bounded from below. Similar results have been obtained in Il'yasov-Silva \cite{kaya} and Silva-Macedo \cite{kayeabi}. 

We observe that $\gamma^*=\Gamma_f((0,\infty))$, where $\Gamma_f((0,\infty))$ is defined in \cite{bobil2}. However, as will be seen on the text our definition is slight different from that one given in  \cite{bobil2} and it allows us to extract more information about the problem \eqref{pq}. As a consequence  we obtain our second result
\begin{thm}\label{AKGLOBAL1}
	 Assume that \eqref{alphabeta} is satisfied and $F(\varphi_1,\psi_1)<0$. If $f_1$ or $f_2$ are satisfied, then for each $\sigma=(\lambda,\mu)\in \gamma^*$, there exists $\varepsilon_\sigma$ such that problem \eqref{pq} has at least a positive solution $(u,v)$ with negative energy for all $\bar{\sigma}\in [\lambda,\lambda+\varepsilon_\sigma)\times[\mu,\mu+\varepsilon_\sigma)$.
\end{thm}

 \begin{figure}[H]
 	\centering\footnotesize   
 	{\def\svgwidth{0.8\linewidth}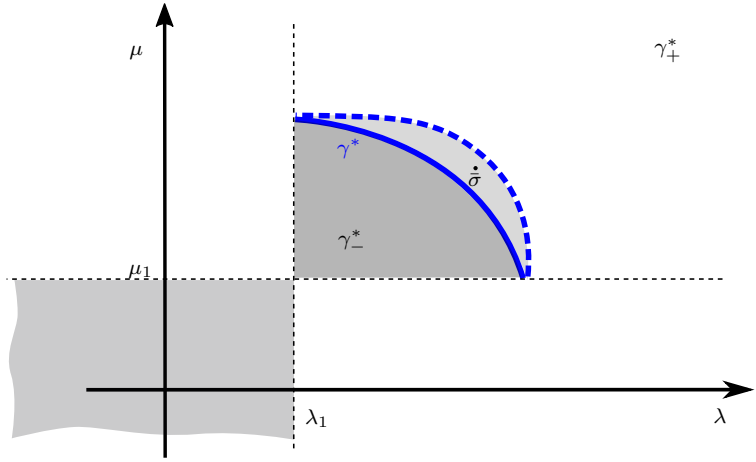}
 	\caption{Positive Solutions For $\sigma$ above the curve $\gamma^*$}
 	\label{8}
 \end{figure}
 
Theorem \ref{AKGLOBAL1} is an improvement on the results of \cite{bobil,bobil2}, since it shows existence of positive solutions to \eqref{pq} when $\sigma$ is above the curve. 

The work is organized as follows: in section \ref{tec} we prove some technical results. In section \ref{Sec} we study the extremal curve $\gamma^*$ . In section \ref{abovepar} we prove the existence of positive solution under the extremal parameter curve while in section \ref{atpar} we prove the existence of a positive solution with negative energy at the extremal parameter curve. In the section \ref{over*} we prove the existence of positive solution on $\gamma^*_+$. In section \ref{sumres} we prove the Theorem \ref{AKGLOBAL} and \ref{AKGLOBAL1}. In section \ref{remarks} we give some remarks on the possible existence of a second branch of positive solutions and a  bifurcation result. In the Appendix we prove an auxiliary result.

\section{Notation and Technical Results}\label{tec}

In this section we establish some notations and results which will be useful in the next sections. Some of the results here can be found in \cite{bobil}. For $\sigma\in \mathbb{R}^2$ and $(u,v)\in W_0^{1,p}(\Omega)\times W_0^{1,q}(\Omega)$, we denote
\begin{equation*}
P_\lambda(u)=\int |\nabla u|^p-\lambda \int|u|^p,\ Q_\mu(v)=\int |\nabla v|^q-\mu \int|v|^q.
\end{equation*}
The Nehari manifold is defined by 
\begin{equation}\label{nehari}
\mathcal{N}_\sigma =\{(u,v)\in  W_0^{1,p}(\Omega)\times W_0^{1,q}(\Omega):\ D\Phi_\sigma(u,v)(u,v)=0  \},
\end{equation}
and it can always be divided into three sets (some of them may be empty) in the following way $\mathcal{N}_\sigma=\mathcal{N}_\sigma^+\cup \mathcal{N}_\sigma^0\cup \mathcal{N}_\sigma^-$ where 
\begin{equation*}\label{N+}
\mathcal{N}_\sigma^+=\{(u,v)\in \mathcal{N}_\sigma:\ P_\lambda(u)<0,\ Q_\mu(v)<0,\ F(u,v)<0 \},
\end{equation*}
\begin{equation*}\label{N0}
\mathcal{N}_\sigma^0=\{(u,v)\in \mathcal{N}_\sigma:\ P_\lambda(u)=Q_\mu(v)=F(u,v)=0 \},
\end{equation*}
and
\begin{equation*}\label{N-}
\mathcal{N}_\sigma^-=\{(u,v)\in \mathcal{N}_\sigma:\ P_\lambda(u)>0,\ Q_\mu(v)>0,\ F(u,v)>0 \}.
\end{equation*}

Observe that all critical points of $\Phi_\sigma$ are contained in $\mathcal{N}_\sigma$. One can see (by studying the signal of the energy functional $\Phi_\sigma$) that the solutions found in \cite{bobil} belongs to $\mathcal{N}_\sigma^-$ if $\sigma\in (-\infty,\lambda_1)\times (-\infty,\mu_1)$ and they belong to $\mathcal{N}_\sigma^+$ if $\sigma=(\lambda,\mu)$ and $\lambda>\lambda_1$, $\mu>\mu_1$. 

In this work we are interested on the sets $\mathcal{N}_\sigma^+$ and $\mathcal{N}_\sigma^0$. When $\mathcal{N}_{\sigma}^+\neq \emptyset$, the Implicit Function Theorem and \eqref{alphabeta} implies that $\mathcal{N}_{\sigma}^+$ is a $C^1$ manifold of codimension two in  $W_0^{1,p}(\Omega)\times W_0^{1,q}(\Omega)$. Let 
$$
S=\{(u,v)\in  W_0^{1,p}(\Omega)\times W_0^{1,q}(\Omega):\ \|u\|_{1,p}=1=\|v\|_{1,q} \},
$$ 
and 
$$
\Theta_{\sigma}=S\cap \left\{ (u,v)\in  W_0^{1,p}(\Omega)\times W_0^{1,q}(\Omega):\ P_\lambda(u)<0,\ Q_\mu(v)<0,\ F(u,v)<0\right\}.
$$
For each $(u,v)\in \Theta_{\sigma}$, there are unique $t_\sigma\equiv t_\sigma(u,v)$, $s_\sigma\equiv s_\sigma(u,v)$ such that $(t_\sigma u,s_\sigma v)\in\mathcal{N}_{\sigma}^+ $. In fact, 
\begin{equation}\label{t}
t_\sigma^{pqd}=\frac{\alpha^{\beta-q}}{\beta^\beta}\frac{|P_\lambda(u)|^{q-\beta}|Q_\mu(v)|^\beta}{|F(u,v)|^q},
\end{equation}
and
\begin{equation}\label{s}
s_\sigma^{pqd}=\frac{\beta^{\alpha-p}}{\alpha^\alpha}\frac{|P_\lambda(u)|^{\alpha}|Q_\mu(v)|^{p-\alpha}}{|F(u,v)|^p},
\end{equation}
where 
$$
d\equiv\frac{\alpha}{p}+\frac{\beta}{q}-1.
$$

Note that from (\ref{alphabeta}), the number $d$ is positive. From this, it follows that $\Theta_{\sigma}$ is diffeomorphic to $\mathcal{N}_{\sigma}^+$ and if $J_{\sigma}:\Theta_{\sigma}\to \mathbb{R}$ is defined by 
\begin{equation}\label{J1}
J_{\sigma}(u,v)\equiv \Phi_{\sigma}(t_\sigma u,s_\sigma v)=-C\frac{|P_\lambda(u)|^{\alpha/(pd)}|Q_\mu(v)|^{\beta/(qd)}}{|F(u,v)|^{1/d}},
\end{equation}
where 
$$
C=\left(\frac{1}{\alpha^{\alpha q}\beta^{\beta q}}\right)^{\frac{1}{pqd}}d,
$$
we obtain that $J_{\sigma}$ is a $C^1$ functional and critical points of $J_{\sigma}$ over $\Theta_{\sigma}$ are critical points of $\Phi_{\sigma}$ in $W_0^{1,p}(\Omega)\times W_0^{1,q}(\Omega)$. 

\begin{prop}\label{criticalpoint} A point $(u,v)\in \Theta_\sigma$ is a critical point to $J_\sigma$ if, and only if $(t_\sigma u,s_\sigma v)$ is a critical point to $\Phi_\sigma$.
\end{prop}

When the Nehari manifold $\mathcal{N}_\sigma^+$ is separated from $\mathcal{N}_\sigma^0$, in the sense that the boundary of $\mathcal{N}_\sigma^+$ does not intersect $\mathcal{N}_\sigma^0$, then one can use standard minimization arguments to show that $\Phi_\sigma$ has a minimizer over $\mathcal{N}_\sigma^+$. The main difficult begins when $\mathcal{N}_\sigma^0$ belongs to the boundary of $\mathcal{N}_\sigma^+$. In that case, we need to study the following extremal parameter (see \cite{ilyasENMM})

\begin{equation}\label{extre}
\sigma^*=\inf_{(u,v)\in  W_0^{1,p}(\Omega)\times W_0^{1,q}(\Omega)}\left\{\max\left\{\frac{1}{\lambda_1}\frac{\int|\nabla u|^p}{\int |u|^p},\frac{1}{\mu_1}\frac{\int|\nabla v|^q}{\int |v|^q}\right\}:\ F(u,v)\ge 0\right\}.
\end{equation}
We say that $(u,v)\in W_0^{1,p}(\Omega)\times W_0^{1,q}(\Omega)$ solves $\sigma^*$ if  $F(u,v)\ge 0$ and 
\begin{equation*}
\sigma^*=\max\left\{\frac{1}{\lambda_1}\frac{\int |\nabla u|^p}{\int |u|^p},\frac{1}{\mu_1}\frac{\int |\nabla v|^q}{\int |v|^q}\right\}.
\end{equation*}	

The next proposition can be found in \cite{bobil} but for completeness we prove it here.

\begin{prop}\label{sigma1} Suppose that $F(\varphi_1,\psi_1)<0$, then $\sigma^*>1$.
\end{prop}

\begin{proof} Suppose on the contrary that $\sigma^*=1$. Standard minimization arguments shows the existence of $(u,v)$ which solves $\sigma^*$. It follows that
	\begin{equation*}
	\frac{1}{\lambda_1}\frac{\int |\nabla u|^p}{\int |u|^p}\le 1,\ \ \frac{1}{\mu_1}\frac{\int |\nabla v|^q}{\int |v|^q}\le 1,
	\end{equation*}
which implies that 	$u=c_1\varphi_1$, $v=c_2\psi_1$, where $c_1,c_2$ are positive constants, however, this contradicts the inequality $F(\varphi_1,\psi_1)<0$.
\end{proof}

In the next lemma we prove that the minimum in \eqref{extre} is attained on the boundary.

\begin{lem}\label{extremal} Assume that \eqref{alphabeta} is satisfied and $F(\varphi_1,\psi_1)<0$. If $f_1$ or $f_2$ are satisfied, then there exists $(u,v)\in W_0^{1,p}(\Omega)\times W_0^{1,q}(\Omega)$ which solves $\sigma^*$.	Moreover, if $(u,v)$ satisfies \eqref{extre} then $F(u,v)=0$.
\end{lem}
\begin{proof} The existence follows from standard minimization arguments. Now, let $(\bar{u},\bar{v})\in W_0^{1,p}(\Omega)\times W_0^{1,q}(\Omega)$ satisfies \eqref{extre}. Without loss of generality we assume that $u\geq 0$ and $v\geq 0$. We want to show that $F(\bar{u},\bar{v})=0$. 
	
	Suppose by contradiction that $F(\bar{u},\bar{v})>0$. Thus 
	$$
	(\bar{u},\bar{v})\in\{(u,v)\in W_0^{1,p}(\Omega)\times W_0^{1,q}(\Omega): F(u,v)>0 \},
	$$ 
	that is an open set. For $\epsilon>0$ define
	$$
	V_\epsilon:=\{(u,v)\in W_0^{1,p}(\Omega)\times W_0^{1,q}(\Omega): \|u-\bar{u}\|_{1,p}<\epsilon \mbox{ and } \|v-\bar{v}\|_{1,q}<\epsilon \},
	$$
	and choose $\epsilon>0$ such that 
	$$
		V_\epsilon\subset \{(u,v)\in W_0^{1,p}(\Omega)\times W_0^{1,q}(\Omega): F(u,v)>0 \}.
	$$
	
	If
	\begin{equation}\label{sig u}
	 \frac{1}{\lambda_1}\frac{\int |\nabla \bar{u}|^p}{\int |\bar{u}|^p}\leq\frac{1}{\lambda_1}\frac{\int |\nabla u|^p}{\int |u|^p},
	\end{equation}
	 for all $u\in\pi_1(V_\epsilon)$, where $\pi_1$ is the projection over $W_0^{1,p}(\Omega)$, then  $\bar{u}\in W_0^{1,p}(\Omega)$ will be a local minimum for $\frac{1}{\lambda_1}\frac{\int |\nabla u|^p}{\int |u|^p}$ in $W_0^{1,p}(\Omega)$ and hence $\sigma^*=1$ which contradicts the Proposition \ref{sigma1}. Therefore there exists $u\in\pi_1(V_\epsilon)$ such that 
	 \begin{equation}\label{u pi1}
	 \frac{1}{\lambda_1}\frac{\int |\nabla u|^p}{\int |u|^p}< \frac{1}{\lambda_1}\frac{\int |\nabla \bar{u}|^p}{\int |\bar{u}|^p}\le \sigma^*.
	 \end{equation}
	 In a similar way there exists  $v\in\pi_2(V_\epsilon)$, where $\pi_2$ is the projection over $W_0^{1,q}(\Omega)$, such that 
	  \begin{equation}\label{v pi2}
	 \frac{1}{\mu_1}\frac{\int |\nabla v|^q}{\int |v|^q}<\frac{1}{\mu_1}\frac{\int |\nabla \bar{v}|^q}{\int |\bar{v}|^q} \leq\sigma^*.
	 \end{equation}
	 Thus $(u,v)\in V_\epsilon$ and
	 $$
	  \max\left\{\frac{1}{\lambda_1}\frac{\int |\nabla u|^p}{\int |u|^p},\frac{1}{\mu_1}\frac{\int |\nabla v|^q}{\int |v|^q}\right\}<\sigma^*,
	  $$
	  which is a contradiction and hence $F(\bar{u},\bar{v})=0$.
	  \end{proof}
	  From the Lemma \ref{extremal} we conclude   
	  \begin{cor}\label{fzero} There holds
	  	\begin{equation*}
	  	\sigma^*=\inf_{u,v}\left\{\max\left\{\frac{1}{\lambda_1}\frac{\int |\nabla u|^p}{\int |u|^p},\frac{1}{\mu_1}\frac{\int |\nabla v|^q}{\int |v|^q}\right\},\ F(u,v)= 0 \right\}.
	  	\end{equation*}
	  \end{cor}
	  
	  Now we study some properties of the solutions of $\sigma^*$.
	  
	  \begin{lem} \label{extremequal} Assume that \eqref{alphabeta} is satisfied and $F(\varphi_1,\psi_1)<0$. If $f_1$ or $f_2$ are satisfied and $(\bar{u},\bar{v})$ solves $\sigma^*$, then 
	  	\begin{equation}\label{iq lev}
	  	\frac{1}{\lambda_1}\frac{\int |\nabla \bar{u}|^p}{\int |\bar{u}|^p}=\frac{1}{\mu_1}\frac{\int |\nabla \bar{v}|^q}{\int |\bar{v}|^q}.
	  	\end{equation}
	  	\end{lem}
	  \begin{proof} Without loss of generality we assume $\bar{u}\geq 0$ and $\bar{v}\geq 0$ and that $\operatorname{int}(\Omega^0\cup \Omega^+)$ is  connected. First observe that for all $(w_1,w_2)\in W:= W_0^{1,p}(\operatorname{int}(\Omega^0\cup \Omega^+))\times W_0^{1,q}(\operatorname{int}(\Omega^0\cup \Omega^+))$ and $t,s\in \mathbb{R}$, there holds $F(\bar{u}+tw_1,\bar{v}+sw_2)\ge 0$. Suppose on the contrary that   \eqref{iq lev} is not true, let's say
	  	  \begin{equation}\label{abcdefg}
	  \sigma^*=\frac{1}{\lambda_1}\frac{\int |\nabla \bar{u}|^p}{\int |\bar{u}|^p}>\frac{1}{\mu_1}\frac{\int |\nabla \bar{v}|^q}{\int |\bar{v}|^q}.
	  \end{equation}

	  We suppose that $\bar{v}$ is choosen in such a way that 	  
	  \begin{equation*}
	  \frac{1}{\mu_1}\frac{\int |\nabla \bar{v}|^q}{\int |\bar{v}|^q}\le \frac{1}{\mu_1}\frac{\int |\nabla v|^q}{\int |v|^q},
	  \end{equation*}
	  for each pair $(u,v)$ of solutions of $\sigma^*$. It follows that $(\bar{u},\bar{v})$ solves the following problems
	  \begin{equation}\label{ex1}
	  \inf \left\{\frac{1}{\lambda_1}\frac{\int |\nabla u|^p}{\int |u|^p}: Q_{\sigma^*}(v)\le 0,\
	 \begin{matrix}
	  u=\bar{u}+t w_1\\ v=\bar{v}+sw_2
	 \end{matrix}, \  (w_1,w_2)\in W\right\},
	  \end{equation}
	  and
	  \begin{equation}\label{ex2}
	  \inf \left\{\frac{1}{\mu_1}\frac{\int |\nabla v|^q}{\int |v|^q}: P_{\sigma^*}(u)\le 0, \
	  \begin{matrix}
	  u=\bar{u}+t w_1\\ v=\bar{v}+sw_2
	  \end{matrix}, \ (w_1,w_2)\in W\right\}.
	  \end{equation}

	 From \eqref{abcdefg} we conclude that  the function 
	 \begin{equation*}
	 W_0^{1,p}(\Omega)\times W_0^{1,q}(\Omega)\ni (u,v)\mapsto\max\left\{\frac{1}{\lambda_1}\frac{\int |\nabla u|^p}{\int |u|^p},\frac{1}{\mu_1}\frac{\int |\nabla v|^q}{\int |v|^q}\right\},
	 \end{equation*}
	 is of class $C^1$ at  $(\bar{u},\bar{v})$. In order to apply the Lagrange's multiplier theorem to the problem $\sigma^*$, we need to show that $DF(\bar{u},\bar{v})$ is surjective. Suppose, ad absurdum, that $DF(\bar{u},\bar{v})$ is not surjective, then for each $(w,\overline{w})\in  W_0^{1,p}(\Omega)\times W_0^{1,q}(\Omega)$, there holds
	 \begin{equation*}
	 DF(\bar{u},\bar{v})(w,\overline{w})=\alpha\int f|\bar{u}|^{\alpha-2}\bar{u}w|\bar{v}|^\beta+\beta\int f|\bar{u}|^\alpha|\bar{v}|^{\beta-2}\bar{v}\overline{w}=0,
	 \end{equation*}
	 which implies that $\operatorname{supp}(\bar{u}\bar{v})\subset \Omega^0$. If $f_1$ is satisfied this is clearly impossible. Thus, assuming $f_2$, let us prove that this is a contradiction. In order to reach a contradiction, we will study the problems \eqref{ex1} and \eqref{ex2}.
	 
	  Indeed, we claim that  $D_uP_{\sigma^*}(\bar{u})$ is surjective in $W$ (which will give us a contradiction). Suppose on the contrary that $D_uP_{\sigma^*}(\bar{u})=0$ in $W$ . Combining the regularity results of \cite{Lieberman92} and the maximum principle of \cite{Vazquez84},  we obtain that $\bar{u}>0$ in $\operatorname{int}(\Omega^0\cup \Omega^+)$ and since $\operatorname{supp}(\bar{u}\bar{v})\subset \Omega^0$ we conclude that $\bar{v}=0$ over $\operatorname{int}( \Omega^+)$. We apply the maximum principle again to conclude  that $D_vQ_{\sigma^*}(\bar{v})$ is surjective and by using the Lagrange's multiplier theorem in \eqref{ex1}, we can find $\nu\neq 0$ such that 
	 \begin{equation*}
	 D_u\left(\frac{1}{\lambda_1}\frac{\int |\nabla \bar{u}|^p}{\int |\bar{u}|^p}\right)=\nu D_vQ_{\sigma^*}(\bar{v}),\ \mbox{in}\ W,
	 \end{equation*}
	  which implies that $ D_vQ_{\sigma^*}(\bar{v})=0$ in $W$ and this is a contradiction. It follows that $D_uP_{\sigma^*}(u)$ is surjective and applying the Lagrange's multiplier theorem on \eqref{ex2}, there exists   $\nu\neq 0$ such that 
	  
	  \begin{equation*}
	  D_v\left(\frac{1}{\mu_1}\frac{\int |\nabla \bar{v}|^p}{\int |\bar{v}|^q}\right)=\nu D_uP_{\sigma^*}(\bar{u}),\ \mbox{in}\ W,
	  \end{equation*}
	 however arguing as above we reach a contradiction.
	 
	  Therefore $DF(u,v)$ is surjective and from the Lagrange's multiplier theorem applied to the minimization problem $\sigma^*$, there exists $\nu\neq 0$ such that for each $(w,\overline{w})\in  W_0^{1,p}(\Omega)\times W_0^{1,q}(\Omega)$, there holds
	 
	 \begin{equation*}
	 \frac{p}{\lambda_1}\frac{\int |\bar{u}|^p\int |\nabla \bar{u}|^{p-2}\nabla \bar{u}\nabla w-\int |\nabla \bar{u}|^p \int |\bar{u}|^{p-2}\bar{u} w}{\left(\int |\bar{u}|^p\right)^2}=\nu DF(\bar{u},\bar{v})(w,\overline{w}).
	 \end{equation*}	 
	 It follows that $\nu=0$ which is a contradiction and hence \eqref{iq lev} holds true.	  
	 \end{proof}

Define 

 \begin{equation*}
 \lambda^*=\sigma^*\lambda_1,\ \mu^*=\sigma^*\mu_1.
 \end{equation*}

\begin{cor}\label{solution} Assume that \eqref{alphabeta} is satisfied and $F(\varphi_1,\psi_1)<0$. If $f_1$ or $f_2$ are satisfied and $(\bar{u},\bar{v})$ solves $\sigma^*$, then there exists $t,s>0$ such that $u=t\bar{u}$ and $v=s\bar{v}$ satisfies 
	
	\begin{equation*}
	\left\{
	\begin{aligned}
	-\Delta_p u &= \lambda^* |u|^{p-2}u+f\alpha |u|^{\alpha-2}|v|^\beta u 
	&\mbox{in}\ \ \Omega, \nonumber\\ 
	-\Delta_q v &= \mu^* |v|^{q-2}v+f\beta |u|^{\alpha}|v|^{\beta-2}v
	&\mbox{in}\ \ \Omega.
	\end{aligned}
	\right.
	\end{equation*}
\end{cor}

\begin{proof} Indeed, observe from the Corollary \ref{fzero} and Lemma \eqref{extremequal} that
	
	\begin{equation*}
	\sigma^*=\inf\left\{\frac{1}{\lambda_1}\frac{\int |\nabla u|^p}{\int |u|^p}:\ Q_{\sigma^*}(v)=0,\ F(u,v)=0\right\}.
	\end{equation*}
	
Arguing as in the proof of the Lemma \ref{extremequal}, we conclude from the Lagrange's multiplier theorem that there exists $\nu_1,\nu_2\neq 0$ such that

\begin{equation*}
D_u\left(\frac{1}{\lambda_1}\frac{\int |\nabla \bar{u}|^p}{\int |\bar{u}|^p}\right)=\nu_1 D_vQ_{\sigma^*}(\bar{v})+\nu_2DF(\bar{u},\bar{v}).
\end{equation*}  	
	
It follows that there exists constants $c_1,c_2\neq 0$ such that

	\begin{equation*}
\left\{
\begin{aligned}
-\Delta_p \bar{u} &= \lambda^* |\bar{u}|^{p-2}\bar{u}+c_1\alpha f|\bar{u}|^{\alpha-2}|\bar{v}|^\beta \bar{u}
&\mbox{in}\ \ \Omega, \nonumber\\ 
-\Delta_q \bar{v} &= \mu^* |\bar{v}|^{q-2}\bar{v}+c_2\beta f|\bar{u}|^{\alpha}|\bar{v}|^{\beta-2}\bar{v}
&\mbox{in}\ \ \Omega.
\end{aligned}
\right.
\end{equation*}

The proof can be completed with a suitable change of variables.	
\end{proof}

\section{Extremal parameters Curve} \label{Sec}

In this section we study the extremal parameters curve to the problem \eqref{pq}. This curve determines a threshold for the applicability of the Nehari manifold method. In fact, as we will see, if the parameter $\sigma$ is above the curve, then $\hat{J}_\sigma=-\infty$, while if it is under the curve then $\hat{J}_\sigma<-\infty$. For $\lambda\in(\lambda_1,\lambda^*]$ and $\mu\in(\mu_1,\mu^*]$ we consider the following functions

\begin{equation*}
\mu_{ext}(\lambda)=\inf\left\{\frac{1}{\mu_1}\frac{\int |\nabla v|^q}{\int |v|^q}:\  P_\lambda(u)\le 0,\ F(u,v)\ge 0 \right\},
\end{equation*}
and
\begin{equation*}
\lambda_{ext}(\mu)=\inf\left\{\frac{1}{\lambda_1}\frac{\int |\nabla u|^p}{\int |u|^p}:\  Q_\mu(v)\le 0,\ F(u,v)\ge 0 \right\}.
\end{equation*}

We prove the following

\begin{lem}\label{extrecurvepa} Assume that \eqref{alphabeta} is satisfied and $F(\varphi_1,\psi_1)<0$. If $f_1$ or $f_2$ are satisfied, then for each $(\lambda,\mu)\in (\lambda_1,\lambda^*]\times (\mu_1,\mu^*]$ there holds $\mu^*< \mu_{ext}(\lambda)<\infty$, $\lambda^*< \lambda_{ext}(\mu)<\infty$ and $\mu_{ext}(\lambda^*)=\lambda_{ext}(\mu^*)=\sigma^*$. Furthermore
	
	\begin{description}
		\item[i)] the problem $\mu_{ext}(\lambda)$ has a minimizer $(u,v)$. Moreover, any minimizer $(u,v)$ of  $\mu_{ext}(\lambda)$ satisfies $F(u,v)=0$, $P_\lambda(u)=0$ and there exists $t,s>0$ such that $(tu,sv)$ solves \eqref{pq};
		\item[ii)] the function $\mu_{ext}$ is continuous and decreasing;
		\item[iii)] assume that $\sigma=(\lambda,\mu)$ satisfies $\lambda\in (\lambda_1,\lambda^*)$ and $\mu^*\le \mu<\mu_{ext}(\lambda)$. Then $P_\lambda(u)\le 0$ and $Q_\mu(v)\le 0$ implies that $F(u,v)<0$.	
		\item[iv)] suppose that $\sigma=(\lambda,\mu)$ satisfies $\lambda\in (\lambda_1,\lambda^*]$ and $\mu_{ext}(\lambda)<\mu$. Then, $\hat{J}_\sigma=-\infty$;
			\item[v)] the problem $\lambda_{ext}(\mu)$ has a minimizer $(u,v)$. Moreover, any minimizer $(u,v)$ of  $\lambda_{ext}(\mu)$ satisfies $F(u,v)=0$, $Q_\mu(v)=0$ and there exists $t,s>0$ such that $(tu,sv)$ solves \eqref{pq};
		\item[vi)] the function $\lambda_{ext}$ is continuous and decreasing;
		\item[vii)] assume that $\sigma=(\lambda,\mu)$ satisfies $\mu\in (\mu_1,\mu^*)$ and $\lambda^*\le \lambda<\lambda_{ext}(\mu)$. Then $P_\lambda(u)\le 0$ and $Q_\mu(v)\le 0$ implies that $F(u,v)<0$.	
		\item[viii)] suppose that $\sigma=(\lambda,\mu)$ satisfies $\mu\in (\mu_1,\mu^*]$ and $\lambda_{ext}(\mu)<\lambda$. Then, $\hat{J}_\sigma=-\infty$;
	\end{description}

\end{lem}

The proof of the Lemma \ref{extrecurvepa} will be a consequence of several results. Arguing as in the Corollary \ref{fzero}, Lemma \ref{extremequal} and Corollary \ref{solution}, one can prove the following

\begin{prop}\label{minimiproblems} For each $(\lambda,\mu)\in (\lambda_1,\lambda^*]\times (\mu_1,\mu^*]$ there holds $\mu^*< \mu_{ext}(\lambda)<\infty$, $\lambda^*< \lambda_{ext}(\mu)<\infty$ and $\mu_{ext}(\lambda^*)=\lambda_{ext}(\mu^*)=\sigma^*$. Moreover, 	
\begin{equation*}
\mu_{ext}(\lambda)=\inf\left\{\frac{1}{\mu_1}\frac{\int |\nabla v|^q}{\int |v|^q}:\  P_\lambda(u)= 0,\ F(u,v)= 0 \right\},
\end{equation*}
and
\begin{equation*}
\lambda_{ext}(\mu)=\inf\left\{\frac{1}{\lambda_1}\frac{\int |\nabla u|^p}{\int |u|^p}:\  Q_\mu(v)= 0,\ F(u,v)= 0 \right\}.
\end{equation*}

Furthermore, there exists $(u,v)$ such that $P_\lambda(u)= 0$, $F(u,v)=0$ and 
\begin{equation*}
\frac{1}{\mu_1}\frac{\int |\nabla v|^q}{\int |v|^q}=\mu_{ext}(\lambda),
\end{equation*}
and if $(u,v)$ solves $\mu_{ext}(\lambda)$, then $P_\lambda(u)= 0$, $F(u,v)=0$ and there exists $t,s>0$ such that $(tu,sv)$ is a solution of \eqref{pq}.

Also, there exists $(u',v')$ such that $Q_{\mu}(v')= 0$, $F(u',v')=0$ and 
\begin{equation*}
\frac{1}{\lambda_1}\frac{\int |\nabla u'|^p}{\int |u'|^p}=\lambda_{ext}(\mu),
\end{equation*}
and if $(u',v')$ solves $\lambda_{ext}(\mu)$, then $Q_\mu(v')=F(u',v')=0$ and there exists $t,s>0$ such that $(tu',sv')$ is a solution of \eqref{pq}.

\end{prop}

From the Proposition \ref{minimiproblems} it follows that (see also \cite{bobil,bobil2})

\begin{cor}\label{solutions} Suppose that $\sigma=(\lambda,\mu)$ satisfies $\lambda\in (\lambda_1,\lambda^*)$ and $\mu^*\le \mu<\mu_{ext}(\lambda)$ or $\mu\in (\mu_1,\mu^*)$ and $\lambda^*\le \lambda<\lambda_{ext}(\mu)$. Then $P_\lambda(u)\le 0$ and $Q_\mu(v)\le 0$ implies that $F(u,v)<0$.	
	
\end{cor}

We prove some properties of the functions  $\mu_{ext}$ and $\lambda_{ext}$

\begin{prop}\label{extecurvepro} The functions $\mu_{ext}$ and $\lambda_{ext}$ are continuous and non increasing.
\end{prop}

\begin{proof} We prove for $\mu_{ext}$ (the other is similar). Observe from the definition that $\mu_{ext}$ is non increasing. Let us prove that 

	\begin{equation*}
	\lim_{\lambda_n\uparrow \lambda}\mu_{ext}(\lambda_n)=\lim_{\lambda_n\downarrow \lambda}\mu_{ext}(\lambda_n)=\mu_{ext}(\lambda).
	\end{equation*} 
	
	Suppose that $\lambda_n\uparrow \lambda\in (\lambda_1,\lambda^*)$ or  $\lambda_n\downarrow \lambda$ as $n\to \infty$. From the Proposition \ref{minimiproblems}, for each $n$, there exists $(u_n,v_n)$ such that $P_{\lambda_n}(u_n)\le 0$, $F(u_n,v_n)=0$ and $\frac{1}{\mu_1}\frac{\int |\nabla v_n|^q}{\int |v_n|^q}=\mu_{ext}(\lambda_n)$. Moreover, we can assume that $\|u_n\|_{1,p}=\|v_n\|_{1,q}=1$ for each $n$. 
	
	 Once $\|u_n\|_{1,p}=\|v_{n}\|_{1,q}=1$, we can assume that 
	$(u_n,v_n)\rightharpoonup (u,v)$ in $W_0^{1,p}(\Omega)\times 	W_0^{1,q}(\Omega)$, $(u_n,v_n)\to (u,v)$ in $L^p(\Omega)\times 	L^q(\Omega)$. Since $P_{\lambda_n}(u_n)\le 0$, we have that  $u\neq 0$. Since $\mu_{ext}$ is non increasing, we can suppose that $\mu_{ext}(\lambda_n)\to I$ as $n\to \infty$. It follows from this that $v\neq 0$. From the weak lower semi-continuity of the norm we conclude that $P_{\lambda}(u)\le \liminf P_{\lambda_n}(u_n)\le 0$ and $F(u,v)=0$ and consequently 	
	\begin{equation*}
	 \mu_{ext}(\lambda)\le \frac{1}{\mu_1}\frac{\int |\nabla v|^q}{\int |v|^q}\le \liminf \frac{1}{\mu_1}\frac{\int |\nabla v_n|^q}{\int |v_n|^q}=I.
	\end{equation*}	

We claim that $I=\mu_{ext}(\lambda)$. Indeed, this is true if $\lambda_n\downarrow \lambda$ as $n\to \infty$ because $\mu_{ext}$ is non increasing. Therefore, let us assume on the contrary that $\lambda_n\uparrow \lambda$ as $n\to \infty$ but however $I>\mu_{ext}(\lambda)$. Given $\varepsilon>0$, choose any $(u_\varepsilon,v_\varepsilon)$ such that $P_{\lambda}(u_\varepsilon)<0$, $F(u_\varepsilon,v_\varepsilon)=0$ and
\begin{equation}\label{contiex1}
\mu_{ext}(\lambda)\le \frac{1}{\mu_1}\frac{\int |\nabla v_\varepsilon|^q}{\int|v_\varepsilon| ^q}< \mu_{ext}(\lambda)+\varepsilon<I.
\end{equation}

Observe that $P_{\lambda_n}(u_\varepsilon)=P_\lambda(u_\varepsilon)+(\lambda-\lambda_n)\int |u_\varepsilon|^p$ for each $n$ and since $\lambda_n\to \lambda$, we conclude that $P_{\lambda_n}(u_\varepsilon)<0$ for sufficienty large $n$. From \eqref{contiex1} it follows that 
\begin{equation*}
\mu_{ext}(\lambda_n)\le \frac{1}{\mu_1}\frac{\int |\nabla v_\varepsilon|^q}{\int|v_\varepsilon| ^q}< \mu_{ext}(\lambda)+\varepsilon<I,
\end{equation*}
which contradicts $I=\lim \mu_{ext}(\lambda_n)$ and henece $I=\mu_{ext}(\lambda)$ for the case where $\lambda_n\uparrow \lambda$. The proof that $\lambda_{ext}$ is decreasing is a consequence of the Proposition \ref{minimiproblems}.

\end{proof}

Let $\sigma=(\lambda,\mu)$ and $\sigma'=(\lambda',\mu')$. We say that $\sigma\le\sigma'$ if $\lambda\le \lambda'$ and $\mu\le \mu'$. If at least one of the inequalities is striclty, we write $\sigma<\sigma'$.

\begin{prop}\label{unboun}
	Suppose that for $\sigma=(\lambda,\mu)$ where $\lambda\in (\lambda_1,\lambda^*]$ and $\mu=\mu_{ext}(\lambda)$ or $\mu\in (\mu_1,\mu^*]$ and $\lambda=\lambda_{ext}(\mu)$, there exists $(u,v)$ such that $P_\lambda(u)=Q_\mu(v)=F(u,v)=0$. Then, for each $\sigma'>\sigma$ there holds
	
	\begin{equation*}
	\hat{J}_{\sigma'}=-\infty.
	\end{equation*}
\end{prop}

\begin{proof} Write $\sigma'=(\lambda',\mu')$. Then $P_{\lambda'}(u)=P_{\lambda}(u)+(\lambda-\lambda')\int |u|^p<0$ and $Q_{\mu'}(v)=Q_{\mu}(v)+(\mu-\mu')\int |v|^q<0$. From the Proposition \ref{minimiproblems}, there exists $t,s>0$ such that 	
		\begin{equation}\label{entendeu}
	\left\{
	\begin{aligned}
	-\Delta_p u &= \lambda |u|^{p-2}u+\alpha t^{\alpha-p}s^\beta f|u|^{\alpha-2}|v|^\beta u
	&\mbox{in}\ \ \Omega, \\ 
	-\Delta_q v &= \mu |v|^{q-2}v+\beta t^\alpha s^{\beta-q} f|u|^{\alpha}|v|^{\beta-2}v
	&\mbox{in}\ \ \Omega.
	\end{aligned}
	\right.
	\end{equation}
	
Choose $(w,\bar{w})\in W_0^{1,p}(\Omega)\times W_0^{1,q}(\Omega)$ such that $-\Delta_p uw- \lambda |u|^{p-2}uw<0$ and $-\Delta_q v\bar{w}- \mu |v|^{q-2}v\bar{w}<0$. It follows from \eqref{entendeu} that 	$D F(u,v)(w,\bar{w})<0$ and consequently, once $F(u,v)=0$, we can find a sequence $(u_n,v_n)\in W_0^{1,p}(\Omega)\times W_0^{1,q}(\Omega)$ such that $F(u_n,v_n)<0$ and $(u_n,v_n)\to (u,v)$ in $ W_0^{1,p}(\Omega)\times W_0^{1,q}(\Omega)$. It follows that for sufficiently large $n$, there exists $c<0$ such that $P_{\lambda'}(u_n)<c$, $Q_{\mu'}(v_n)<c$ and hence, from \eqref{J1} we have that	
	\begin{equation*}
	\lim_{n\to \infty} J_{\sigma'} \left(\frac{u_n}{\|u_n\|_{1,p}},\frac{v_n}{\|v_n\|_{1,q}}\right)= -\infty.
	\end{equation*}

\end{proof}

From the Propositions \ref{extecurvepro} and \ref{unboun} we obtain

\begin{cor}\label{unbounini} Suppose that $\sigma=(\lambda,\mu)$ satisfies $\lambda\in (\lambda_1,\lambda^*]$ and $\mu_{ext}(\lambda)<\mu$ or $\mu\in (\mu_1,\mu^*]$ and $\lambda_{ext}(\mu)<\lambda$. Then, $\hat{J}_\sigma=-\infty$.
\end{cor}

Now we can prove the Lemma \ref{extrecurvepa}

\begin{proof}[Proof of Lemma \ref{extrecurvepa}] It follows from the Propositions \ref{minimiproblems}, \ref{extecurvepro} and the Corollaries 	\ref{solutions}, \ref{unbounini}.
\end{proof}
\section{Minimizers to $\hat{J}_\sigma$ under the extremal parameter curve} \label{abovepar}

We denote the extremal parameters curve studied in Section \ref{Sec} by

\begin{equation}\label{extcruvegamma}
\gamma^*=\{\sigma=(\lambda,\mu):\  \lambda\in (\lambda_1,\lambda^*],\ \mu=\mu_{ext}(\lambda)\ \mbox{or}\ \mu\in (\mu_1,\mu^*],\ \lambda=\lambda_{ext}(\mu)\}.
\end{equation}

This curve determines two sets over the parameter plane $(\lambda,\mu)$ in the following way

\begin{equation}\label{extcruvegamma1}
	\gamma^*_-=\{\sigma=(\lambda,\mu):\  \lambda\in (\lambda_1,\lambda^*],\ \mu<\mu_{ext}(\lambda)\ \mbox{or}\ \mu\in (\mu_1,\mu^*],\ \lambda<\lambda_{ext}(\mu) \},
\end{equation}
and
\begin{equation}\label{extcruvegamma2}
\gamma^*_+=\{\sigma=(\lambda,\mu):\  \lambda\in (\lambda_1,\lambda^*],\ \mu>\mu_{ext}(\lambda)\ \mbox{or}\ \mu\in (\mu_1,\mu^*],\ \lambda>\lambda_{ext}(\mu)\ \mbox{or}\ \lambda>\lambda^*,\ \mu>\mu^* \}.
\end{equation}

For $\sigma\in \gamma^*_-$, consider the following constrained minimzation problem

\begin{equation}\label{min}
\hat{J}_\sigma=\inf \{J_\sigma (u,v):\ (u,v)\in \Theta_\sigma\}.
\end{equation}

Also consider

\begin{equation}\label{minimizersset}
\mathcal{S}_\sigma=\{(u,v)\in \Theta_\sigma:\ \hat{J}_{\sigma}= J_{\sigma}(u,v)\}.
\end{equation}

The following result can be found in \cite{bobil2} (remember that $\gamma^*=\Gamma_f((-,\infty))$. We state it here for completeness and observe that its proof is a consequence of the Lemma \ref{extrecurvepa}.

\begin{lem}\label{globalmin} Suppose that $\sigma\in \gamma^*_-$. Then, $\mathcal{S}_\sigma\neq \emptyset$. Moreover, for each $(u,v)\in \mathcal{S}_\sigma$ the pair $(t_\sigma u,s_\sigma v)$ solves the equation \eqref{pq}.
\end{lem}

\begin{cor}\label{decrea} Let $\sigma=(\lambda,\mu),\sigma'=(\lambda',\mu')$ and suppose that $\sigma,\sigma' \in \gamma^*_-$ and $\sigma\le\sigma'$. Then $\hat{J}_\sigma\le \hat{J}_{\sigma'}$. Moreover, if  $\sigma<\sigma'$, then $\hat{J}_{\sigma'}< \hat{J}_\sigma$.
\end{cor}

\begin{proof} Choose any $(u,v)\in \mathcal{S}_{\overline{\sigma}}$ and observe that $|P_{\lambda'}(u)|\ge |P_{\lambda}(u)|$ and $|Q_{\mu'}(v)|\ge |Q_{\sigma}(v)|$. Therefore
	
	\begin{align*}
	\hat{J}_{\sigma'}&=J_{\sigma'}(u,v) \\
	&=-C\frac{|P_{\lambda'}(u)|^{\alpha/(pd)}|Q_{\mu'}(v)|^{\beta/(qd)}}{|F(u,v)|^{1/d}} \\
	&\le -C\frac{|P_{\lambda}(u)|^{\alpha/(pd)}|Q_{\mu}(v)|^{\beta/(qd)}}{|F(u,v)|^{1/d}} \\
	&=\hat{J}_{\sigma}.
	\end{align*}
	
	To prove that $\hat{J}_\sigma< \hat{J}_{\sigma'}$ if 	$\sigma<\sigma'$ just observe that $|P_{\lambda'}(u)|> |P_{\lambda}(u)|$ or $|Q_{\mu'}(v)|> |Q_{\mu}(v)|$.
	
\end{proof}
\section{Minimizers to $\hat{J}_\sigma$ at the extremal parameter curve}\label{atpar}

In this section we prove the following:

\begin{lem}\label{minextre} Suppose that $\sigma\in \gamma^*$, then $\mathcal{S}_\sigma\neq \emptyset$. Moreover, 
	\begin{description}
		\item[i)] if $\sigma=(\lambda,\mu_{ext}(\lambda))$, then there exists $(\nu_\sigma,\bar{\nu}_\sigma)\in (\lambda_1,\lambda)\times (\mu_1,\mu_{ext}(\lambda))$ such that $P_{\nu_\sigma}(u)\le 0$ and $Q_{\bar{\nu}_\sigma}(v)\le 0$ for each $(u,v)\in \mathcal{S}_\sigma$;
		\item[ii)] if $\sigma=(\lambda_{ext}(\mu),\mu)$, then there exists $(\nu_\sigma,\bar{\nu}_\sigma)\in (\lambda_1,\lambda_{ext}(\mu))\times (\mu_1,\mu)$ such that $P_{\nu_\sigma}(u)\le 0$ and $Q_{\bar{\nu}_\sigma}(v)\le 0$ for each $(u,v)\in \mathcal{S}_\sigma$.
	\end{description} 
\end{lem}

The items of Lemma \ref{minextre}  says that the minimizers of $\Phi_\sigma$ over the Nehari manifold $\mathcal{N}_\sigma^+$ are separated from the Nehari set $\mathcal{N}_\sigma ^0$. This will be important later to show existence of solutions near the extremal parameter curve. We divide the proof of Lemma \ref{minextre} in two propositions:

\begin{prop}\label{extre1}
	Suppose that $\sigma=(\lambda,\mu)\in \gamma^*$. Assume that $\sigma_n\in \gamma^*_-$ and $\sigma_n\to \sigma$ as $n\to \infty$. Then, $\hat{J}_{\sigma_n}\to \hat{J}_\sigma$ as $n\to \infty$ and $ \hat{J}_\sigma>-\infty$. Moreover, $\mathcal{S}_\sigma\neq \emptyset$.
\end{prop}

\begin{proof} Define $\sigma_n=(\lambda_n,\mu_n)$. From the Lemma \ref{globalmin}, for each $n$, choose $(u_n,v_n)\in \mathcal{S}_{\sigma_n}$ such that

	\begin{equation}\label{pq1}
	\left\{
	\begin{aligned}
	-\Delta_p u_{n} &= \lambda_n {u_n}^{p-1}+\alpha t_n^{\alpha-p}s_n^\beta fu_{n}^{\alpha-1}v_{n}^\beta  
	&\mbox{in}\ \ \Omega \\ 
	-\Delta_q v_{n} &= \mu_n v_{n}^{q-1}+\beta t_n^{\alpha}s_n^{\beta-q} fu_{n}^{\alpha}u_{n}^{\beta-1}
	&\mbox{in}\ \ \Omega 
	\end{aligned}
	\right.
	\end{equation}
	where  $t_n=t_{\sigma_n}(u_n,v_n)$ and $s_n=s_{\sigma_n}(u_n,v_n)$.

	Once $\|u_n\|_{1,p}=\|v_{n}\|_{1,q}=1$, we can assume that 
	$(u_n,v_n)\rightharpoonup (u,v)$ in $W_0^{1,p}(\Omega)\times 	W_0^{1,q}(\Omega)$, $(u_n,v_n)\to (u,v)$ in $L^p(\Omega)\times 	L^q(\Omega)$. Since $P_{\lambda_n}(u_n),H_{\mu_n}(v_n)<0$, we have that  $u,v\neq 0$.
	
	From the Proposition \ref{decrea}, we may assume that there exists a negative constant $c$ such that $c>\hat{J}_{\sigma_n}$ for each $n$. We claim that  $\hat{J}_{\sigma_n}$ is bounded. Indeed suppose on the contrary that up to a subsequence $\hat{J}_{\sigma_n}\to-\infty$ as $n\to \infty$, then from (\ref{J1}) we obtain that $F(u,v)=\lim F(u_n,v_n)=0$. For $(u,v)\in \Theta_{\sigma}$ observe from \eqref{t}, \eqref{s} and \eqref{J1} that
	\begin{equation}\label{J2}
	J_{\sigma_n}(u_n,v_n)=-C_1t_n^p|P_{\lambda_n}(u_n)|,
	\end{equation}
	and
	
	\begin{equation}\label{J3}
	J_{\sigma_n}(u_n,v_n)=-C_2s_n^q|Q_{\mu_n}(v_n)|,
	\end{equation}
	where $C_1,C_2>0$.
	
	Once $|P_\lambda(u)|,|Q_\mu(v)|$ are bounded in $\Theta_{\sigma}$, it follows from (\ref{J2}) and (\ref{J3}) that $t_n,s_n\to \infty$ as $n\to \infty$, however, from (\ref{pq1}) and $\alpha> p$, we conclude that $fu^{\alpha-1}v^\beta=0$ a.e. in $\Omega$. Observe that $F(u,v)=0$ and from the Lemma \ref{extrecurvepa} it follows that $P_{\lambda_n}(u_n)$, $Q_{\mu_n}(v_n)$ converge to zero as $n\to \infty$ and $(u,v)$ solves $\sigma^*$, however, arguing as in the Lemma \ref{extremequal} we reach a contradiction.
	
	Therefore there exists $C<0$ such that  $c>\hat{J}_{\sigma_n}>C$ for each $n$, and from \eqref{J1}, \eqref{J2} and \eqref{J3}, we may assume that $t_n,s_n\to t,s$ with $t,s\in (0,\infty)$. From the $S^+$ property of $-\Delta_p,-\Delta_q$ (see \cite{drabekMilota}), and (\ref{pq1}) we obtain that $(u_n,v_n)\to (u,v)$ in $W_0^{1,p}(\Omega)\times W_0^{1,q}(\Omega)$. It follows that $P_{\lambda}(u)<0$, $Q_\mu(v)<0$, $F(u,v)<0$ and hence $\hat{J}_{\sigma_n}\to I=J_\sigma (u,v)$ as $n\to \infty$, where $c\ge I\ge C$.
	
	 Now we claim that $I=\hat{J}_\sigma$. Suppose ad absurdum that $I>\hat{J}_\sigma$. Choose $(w,\overline{w})\in\Theta_\sigma$ such that $I>J_\sigma(w,\overline{w})\ge \hat{J}_\sigma$. Observe from \eqref{J1} that
	
	\begin{equation*}
	\lim_{n\to \infty}(J_{\sigma_n}(w,\overline{w})-J_\sigma (w,\overline{w}))=0,
	\end{equation*}
	and therefore there exists $n$ such that $\hat{J}_{\sigma_n}>I>J_{\sigma_n}(w,\overline{w})$, which is clearly a contradiciton and consequently 
	
	\begin{equation*}
	\lim_{n\to \infty}\hat{J}_{\sigma_n}=J_\sigma(u,v)=I=\hat{J}_\sigma.
	\end{equation*}	
\end{proof}

Now we study the sets $\mathcal{S}_\sigma$ when $\sigma\in \gamma^*$.

\begin{prop}\label{Ssigma} Suppose that $(\lambda,\mu)\in (\lambda_1,\lambda^*]\times (\mu_1,\mu^*]$. There holds
	
	\begin{description}
		\item[i)] if $\sigma=(\lambda,\mu_{ext}(\lambda))$, then there exists $(\nu_\sigma,\bar{\nu}_\sigma)\in (\lambda_1,\lambda)\times (\mu_1,\mu_{ext}(\lambda))$ such that $P_{\nu_\sigma}(u)\le 0$ and $Q_{\bar{\nu}_\sigma}(v)\le 0$ for each $(u,v)\in \mathcal{S}_\sigma$;
		\item[ii)] if $\sigma=(\lambda_{ext}(\mu),\mu)$, then there exists $(\nu_\sigma,\bar{\nu}_\sigma)\in (\lambda_1,\lambda_{ext}(\mu))\times (\mu_1,\mu)$ such that $P_{\nu_\sigma}(u)\le 0$ and $Q_{\bar{\nu}_\sigma}(v)\le 0$ for each $(u,v)\in \mathcal{S}_\sigma$.
\end{description} 
\end{prop}

\begin{proof} We prove $\mathbf{i)}$ (the proof of $\mathbf{ii)}$ being similar). Indeed, fix $\sigma=(\lambda,\mu_{ext}(\lambda))$ and suppose on the contrary that for each $(\nu_\sigma,\bar{\nu}_\sigma)\in (\lambda_1,\lambda)\times (\mu_1,\mu_{ext}(\lambda))$, there exists $(u_{\sigma},v_\sigma)\in \mathcal{S}_\sigma$ such that $P_{\nu_\sigma}(u_\sigma)> 0$ or $Q_{\overline{\nu}_\sigma}(v_\sigma)> 0$. Therefore, given a sequence $(\nu_n,\bar{\nu}_n)$ such that $\nu_n\to \lambda$ and $\bar{\nu}_n\to \mu_{ext}(\lambda)$ as $n\to \infty$,	 there exist sequence $(u_n,v_n)\in  \mathcal{S}_{\sigma}$ such that $P_{\nu_n}(u_n)> 0$ or $Q_{\overline{\nu}_n}(v_n)> 0$ for each $n$. Once $\|u_n\|_{1,p}=\|v_{n}\|_{1,q}=1$, we can assume that 
$(u_n,v_n)\rightharpoonup (u,v)$ in $W_0^{1,p}(\Omega)\times 	W_0^{1,q}(\Omega)$, $(u_n,v_n)\to (u,v)$ in $L^p(\Omega)\times 	L^q(\Omega)$. Since $P_{\lambda}(u_n),Q_{\mu_{ext}(\lambda)}(v_n)<0$, we have that  $u,v\neq 0$.

Note that for all $n$ it follows that $0>P_\lambda(u_n)=P_{\nu_n}(u_n)+(\nu_n-\lambda)\int |u_n|^p\ge (\nu_n-\lambda)\int |u_n|^p$ and $0>Q_\mu(v_n)=Q_{\bar{\nu}_n}(v_n)+(\bar{\nu}_n-\lambda)\int |v_n|^q\ge (\bar{\nu}_n-\lambda)\int |v_n|^q$, which implies that at least one of the sequences $P_\lambda(u_n),Q_\mu(v_n)$ converge to zero as $n\to \infty$.  From \eqref{J1} and the Proposition \ref{extre1} it follows $F(u_n,v_n)\to 0$ as $n\to \infty$. It follows that $P_\lambda(u)=Q_{\mu_{ext}}(v)=F(u,v)=0$ and $\operatorname{supp}(uv)\subset \Omega^0$. Arguing as in the Proposition \ref{extre1} we reach a contradiction and therefore, there exists $(\nu_\sigma,\bar{\nu}_\sigma)\in (\lambda_1,\lambda)\times (\mu_1,\mu_{ext}(\lambda))$ such that $P_{\nu_\sigma}(u)\le 0$ and $Q_{\bar{\nu}_\sigma}(v)\le 0$ for each $(u,v)\in \mathcal{S}_\sigma$.

\end{proof}

Now we prove the Lemma \ref{minextre}. 

\begin{proof}[Proof of the Lemma \ref{minextre}] It follows from the Propositions \ref{extre1} and \ref{Ssigma}.

\end{proof}
\section{Local minimizers for $J_\sigma$ when $\sigma\in\gamma^*_+$}\label{over*}

Assume that $\sigma=(\lambda,\mu)\in \gamma^*_+$ and $\omega=(\nu,\bar{\nu})\in \gamma^*_-$ with $\nu<\lambda$ and $\bar{\nu}<\mu$. Define

\begin{equation*}
\Theta_{\sigma,\omega}=\{(u,v)\in \Theta_\sigma:\ P_\nu(u)< 0,\ Q_{\overline{\nu}}(v)< 0\},
\end{equation*}
and the closure of $\Theta_{\sigma,\omega}$ with respect to the norm topology

\begin{equation*}
\overline{\Theta}_{\sigma,\omega}=\{(u,v)\in \Theta_\sigma:\ P_\nu(u)\le 0,\ Q_{\overline{\nu}}(v)\le 0\},
\end{equation*}
and observe that $\overline{\Theta}_{\sigma,\omega}\subset \Theta_\sigma$. Consider the following constrained minimization problem

\begin{equation*}
\hat{J}_{\sigma,\omega}=\inf\{J_\sigma(u,v):\ (u,v)\in \overline{\Theta}_{\sigma,\omega}\},
\end{equation*}
and
\begin{equation*}
\mathcal{S}_{\sigma,\omega}=\{(u,v)\in \overline{\Theta}_{\sigma,\omega}:\ \hat{J}_{\sigma,\omega}=J_\sigma(u,v)\}.
\end{equation*}

We are interested in the minimizers of $\hat{J}_{\sigma,\omega}$ which belongs to the interior of $\overline{\Theta}_{\sigma,\omega}$, therefore, let us consider the following set

\begin{equation*}
\mathring{\mathcal{S}}_{\sigma,\omega}=\{(u,v)\in \Theta_{\sigma,\omega}:\ \hat{J}_{\sigma,\omega}=J_\sigma(u,v)\}.
\end{equation*}

Observe from the Proposition \ref{unboun} that $\hat{J}_\sigma=-\infty$ if $\sigma \in \gamma^*_+$. However, at least for those parameters which lies near $\gamma^*$, we will see as a consequence of the Lemma \ref{minextre} that $\Phi_\sigma$ still has local minimizers over the Nehari manifold $\mathcal{N}_\sigma^+$.

\begin{lem}\label{locmin} For each $\sigma=(\lambda,\mu)\in \gamma^*$, choose $(\nu_\sigma,\bar{\nu}_\sigma)$ as in the Lemma \ref{minextre}. Suppose that $(\nu_\sigma,\bar{\nu}_\sigma)<\omega\in \gamma^*_-$, then there exists $\varepsilon_\sigma>0$ such that for each $\bar{\sigma}\in (\lambda,\lambda+\varepsilon_\sigma)\times (\mu,\mu+\varepsilon_\sigma)$ there holds $\mathring{\mathcal{S}}_{\bar{\sigma},\omega}\neq \emptyset$.
\end{lem}

To prove Lemma \ref{locmin} we need to study the function $F$ over the sets $ \overline{\Theta}_{\sigma,\omega}$.

\begin{prop}\label{Fbounded} Assume that $\sigma=(\lambda,\mu)\in \gamma^*_+$ and $\omega=(\nu,\bar{\nu})\in \gamma^*_-$ with $\nu<\lambda$, $\bar{\nu}<\mu$. Then, there exists a constant $c_{\sigma,\omega}<0$ such that $F(u,v)<c_{\sigma,\omega}$ for each $(u,v)\in \overline{\Theta}_{\sigma,\omega}$.
\end{prop}

\begin{proof} Suppose on the contrary that there exists $(u_n,v_n)\in \overline{\Theta}_{\sigma,\omega}$ such that $F(u_n,v_n)\to 0$ as $n\to \infty$.
Once $\|u_n\|_{1,p}=\|v_{n}\|_{1,q}=1$, we can assume that 
$(u_n,v_n)\rightharpoonup (u,v)$ in $W_0^{1,p}(\Omega)\times 	W_0^{1,q}(\Omega)$, $(u_n,v_n)\to (u,v)$ in $L^p(\Omega)\times 	L^q(\Omega)$. Since $P_{\lambda}(u_n),H_{\mu}(v_n)<0$, we have that  $u,v\neq 0$. Moreover $F(u,v)=0$. From the weak lower semi-continuity of the norm we obtain that $P_\nu(u)\le \liminf P_\nu(u_n)\le 0$ and $Q_{\overline{\nu}}(v)\le \liminf Q_{\overline{\nu}}(v_n)\le 0$. Once $\omega=(\nu,\bar{\nu})\in \gamma^*_-$ we get a contradiction with the definition of $\gamma^*$. Therefore, for each  $\sigma=(\lambda,\mu)\in \gamma^*_+$ and $\omega=(\nu,\bar{\nu})\in \gamma^*_-$ with $\nu<\lambda$, $\bar{\nu}<\mu$, there exists a constant $c_{\sigma,\omega}$ such that $F(u,v)<c_{\sigma,\omega}$ for each $(u,v)\in \overline{\Theta}_{\sigma,\omega}$. 

\end{proof}

From the definition of $J_\sigma$ we conclude 

\begin{cor}\label{locaminwelldefined} Suppose that $\sigma=(\lambda,\mu)\in \gamma^*_+$ and $\omega=(\nu,\bar{\nu})\in \gamma^*_-$ with $\nu<\lambda$, $\bar{\nu}<\mu$, then $\hat{J}_{\sigma,\omega}>-\infty$.
\end{cor}

From the Corollary \ref{locaminwelldefined} and arguing as in the Section 6 of \cite{bobil} we obtain

\begin{prop}\label{existence} Suppose that $\sigma=(\lambda,\mu)\in \gamma^*_+$ and $\omega=(\nu,\bar{\nu})\in \gamma^*_-$ with $\nu<\lambda$, $\bar{\nu}<\mu$. Then $\mathcal{S}_{\sigma,\omega}\neq \emptyset$.
\end{prop}

\begin{rem}\label{decrea1} From the Proposition \ref{existence}, we conclude that the Proposition \ref{decrea} also holds true fro $\hat{J}_{\sigma,\omega}$.
\end{rem}

For $\sigma=(\lambda,\mu)\in \gamma^*_+$ and $\omega=(\nu,\bar{\nu})\in \gamma^*_-$ with $\nu<\lambda$, $\bar{\nu}<\mu$ define

\begin{equation*}
\mathcal{S}_{\sigma,\omega}^\partial=\mathcal{S}_{\sigma,\omega}\setminus \mathring{\mathcal{S}}_{\sigma,\omega}.
\end{equation*}

Observe that $\mathcal{S}_{\sigma,\omega}= \mathring{\mathcal{S}}_{\sigma,\omega}\cup \mathcal{S}_{\sigma,\omega}^\partial$. Now we can prove Lemma \ref{locmin}.

\begin{proof}[Proof of Lemma \ref{locmin}] For $\varepsilon\ge 0$ define $\mathcal{A}_\varepsilon =(\lambda,\lambda+\varepsilon)\times(\mu,\mu+\varepsilon)$. We claim that there exists $\varepsilon>0$ such that for all $\bar{\sigma}\in \mathcal{A}_{\varepsilon}$, there holds $\mathcal{S}_{\bar{\sigma},\omega}^\partial=\emptyset $. Indeed, suppose on the contrary that for each $n\in\mathbb{N}$, there exists $\sigma_n\in \mathcal{A}_{1/n}$ such that $\mathring{\mathcal{S}}_{\sigma_n,\omega}=\emptyset$. It follows from the Proposition \ref{existence} that for each $n$ there holds $\mathcal{S}_{\sigma_n,\omega}^\partial \neq \emptyset$.
	
	Choose a sequence $(u_n,v_n)\in \mathcal{S}_{\sigma_n,\omega}^\partial$. From the Proposition \ref{uni} we conclude that $(u_n,v_n)\to (u,v)\in \mathcal{S}_\sigma$. However, $P_{\nu}(u_n)=Q_{\overline{\nu}}(v_n)=0$ implies that $P_{\nu}(u)=Q_{\overline{\nu}}(v)=0$, which contradicts the Proposition \ref{Ssigma}.
	
\end{proof}

\section{Proof of the Theorems \ref{AKGLOBAL} and \ref{AKGLOBAL1}} \label{sumres}

In this section we prove our main results.

\begin{proof}[Proof of the Theorem \ref{AKGLOBAL}] Denote $\sigma=(\lambda,\mu)$. We start with $\mathbf{i)}$: from the Lemma \ref{extrecurvepa} there exists $(u,v)\in W_0^{1,p}(\Omega)\times W_0^{1,q}(\Omega)$ and $t,s>0$ such that $(tu,sv)$ is a solution of \eqref{pq}. Observe that $P_\lambda(|u|)=P_\lambda(u)=0$, $Q_\mu(|v|)=P_\lambda(v)=0$ and $F(|u|,|v|)=F(u,v)=0$. It follows that $(|u|,|v|)$ solves one of the problems $\mu_{ext}(\lambda)$ or $\lambda_{ext}(\mu)$ and consequently, from the Lemma \ref{extrecurvepa} $(t|u|,s|v|)$ is a solution to \eqref{pq} with $\Phi_\sigma(|u|,|v|)=0$. Arguing as in \cite{bobil}, we conclude that $u$ and $v$ are positive.

$\mathbf{ii)}$: suppose first that $\sigma\in \gamma^*_-$. In this case, the boundedness of $J_\sigma$ follows from $\mathbf{iii)}$, $\mathbf{vii)}$ of the Lemma \ref{extrecurvepa} and \eqref{J1}. If $\sigma\in \gamma^*$, the proof follows from the Lemma \ref{minextre}. 

$\mathbf{iii)}$:	it is a consequence of $\mathbf{iv)}$, $\mathbf{viii)}$ of the Lemma \ref{extrecurvepa}.
	
\end{proof}

\begin{proof}[Proof of the Theorem \ref{AKGLOBAL1}] From the Lemma \ref{minextre}, for each $\sigma\in \gamma^*$, there exists $(u,v)\in \mathcal{S}_\sigma$. Note that $(|u|,|v|)\in \mathcal{S}_\sigma$. It follows from the Proposition \ref{criticalpoint} that there exists $t,s>0$ such that $(t|u|,s|v|)$ is a solution of \eqref{pq}. Arguing as in \cite{bobil}, we conclude that $u$ and $v$ are positive.

Now we prove existence of solutions in $\gamma^*_+$.	From the Lemma \ref{locmin}, if $(\lambda_\sigma,\mu_\sigma)<\omega\in\gamma^*_-$, then there exists $\varepsilon_\sigma>0$ such that for each $\bar{\sigma}\in (\lambda,\lambda+\varepsilon_\sigma)\times (\mu,\mu+\varepsilon_\sigma)$ there holds $\mathring{\mathcal{S}}_{\bar{\sigma},\omega}\neq \emptyset$. Choose any $(u,v)\in \mathring{\mathcal{S}}_{\bar{\sigma},\omega}$ and observe that $(|u|,|v|)\in \mathring{\mathcal{S}}_{\bar{\sigma},\omega}$. It follows from the Proposition \ref{criticalpoint} that there exists $t,s>0$ such that $(t|u|,s|v|)$ is a solution of \eqref{pq}. Arguing as in \cite{bobil}, we conclude that $u$ and $v$ are positive.

\end{proof}	
	
\section{Remarks Concerning the Existence of a Second Branch of Positive Solutions}\label{remarks}

We observe from the Theorems \ref{AKGLOBAL} and \ref{AKGLOBAL1} that for $\sigma\in \gamma^*$, the problem \eqref{pq} has two distinct positive solutions, let's say, one is $(u,v)$ for which $\Phi_\sigma(u,v)<0$ and the other one is $(\bar{u},\bar{v})$ for which $\Phi_\sigma(\bar{u},\bar{v})=0$. This seems to suggest the existence of a second branch of positive solutions for $\sigma \in \gamma^+$. Moreover, it appears that these two branches must connect somewhere over the set $\gamma^+$.

In fact, for $\sigma\in \gamma^+$ as in the Theorem \ref{AKGLOBAL1}, one can see that the functional $J_\sigma$ has a mountain pass geometry, however, the same technique used in \cite{kaya} to prove the existence of a second branch of positive solutions can not be applied here because the geometry of the fibering maps is different. Indeed, in our case here, the critical points of the fibering maps are saddle points and this creates an obstacle to show the mountain pass geometry for the functional $\Phi_\sigma$.

\appendix

\section{}

\begin{prop}\label{uni}  For each $\sigma=(\lambda,\mu)\in \gamma^*$, choose $(\nu_\sigma,\bar{\nu}_\sigma)$ as in the Lemma \ref{minextre}. Suppose that $(\nu_\sigma,\bar{\nu}_\sigma)<\omega\in \gamma^*_-$. Let $\sigma_n\in \mathcal{A}_{1/n}$ and assume that $\sigma_n\to \sigma\in \mathcal{A}_0$,  then $\hat{J}_{\sigma_n,\omega}\to \hat{J}_{\sigma,\omega}$ and $\hat{J}_{\sigma,\omega}=\hat{J}_{\sigma}$. Moreover, if $(u_n,v_n)\in \mathcal{S}_{\sigma_n,\omega}$ then $(u_n,v_n)\to (u,v)\in \mathcal{S}_\sigma$. Arguing as in \cite{bobil,bobil2} we conclude that $u,v$ are positive.
	
\end{prop}

\begin{proof}  Write $\sigma=(\lambda,\mu)$. Once $\|u_n\|_{1,p}=\|v_{n}\|_{1,q}=1$, we can assume that 
	$(u_n,v_n)\rightharpoonup (u,v)$ in $W_0^{1,p}(\Omega)\times 	W_0^{1,q}(\Omega)$, $(u_n,v_n)\to (u,v)$ in $L^p(\Omega)\times 	L^q(\Omega)$. Since $P_{\lambda_n}(u_n),H_{\mu_n}(v_n)< 0$, we have that  $u,v\neq 0$. We can assume from the Remark \ref{decrea1} that $\hat{J}_{\sigma_n,\omega}\to I$ as $n\to \infty$. We claim that $(u_n,v_n)\to (u,v)$ in $W_0^{1,p}(\Omega)\times 	W_0^{1,q}(\Omega)$. Indeed, if this is not true, then from the weak lower semi-continuity of the norm, we have that at least one of the inequalities $P_\lambda(u)\le \liminf P_{\lambda_n}(u_n)$ and $Q_\mu(v)\le \liminf Q_{\mu_n}(v_n)$ is strictly, let's say $P_\lambda(u)<\liminf P_{\lambda_n}(u_n)$. It follows that 
	
	\begin{align*}
	J_\sigma \left(\frac{u}{\|u\|_{1,p}},\frac{v}{\|v\|_{1,q}}\right)&=-C\frac{|P_\lambda(u)|^{\alpha/(pd)}|Q_\mu(v)|^{\beta/(qd)}}{|F(u,v)|^{1/d}} \\
	&<\liminf J_{\sigma_n}(u_n,v_n) \\
	&=\lim \hat{J}_{\sigma_n,\omega} \\
	&=I,
	\end{align*}	
	and therefore $\hat{J}_\sigma <I$.

Now choose any $(u,v)\in \mathcal{S}_\sigma$ and observe from the Proposition \ref{Ssigma} that $(u,v)\in \Theta_{\sigma_n,\omega}$ for all $n$. It follows that  $J_{\sigma_n}(u,v)\to J_{\sigma}(u,v)=\hat{J}_\sigma<I$ as $n\to \infty$, hence, given $\delta>0$, there exists $N>0$ such that for $n>N$ there holds $\hat{J}_{\sigma_n,\omega}\le J_{\sigma_n}(u,v)<I-\delta$, which is a contradiction since $\hat{J}_{\sigma_n,\omega}\to I$ as $n\to \infty$. It follows that $(u_n,v_n)\to (u,v)$ in $W_0^{1,p}(\Omega)\times 	W_0^{1,q}(\Omega)$, $(u,v)\in \mathcal{S}_{\sigma,\omega}$ and $I=\hat{J}_{\sigma,\omega}$. From the Proposition \ref{Ssigma} we have that $\hat{J}_{\sigma,\omega}=\hat{J}_{\sigma}$
	
\end{proof}

%%%%%%%%%%%%%%%%%%%%%%%%%%%%%%%%%%%%%%%%


\bibliographystyle{amsplain}
\bibliography{Ref}

\end{document}

%% file: fig-10.pdf_tex
%% Creator: Inkscape inkscape 0.48.5, www.inkscape.org
%% Accompanies image file 'fig-10.pdf' (pdf, eps, ps)
\begingroup%
  \makeatletter%
  \providecommand\color[2][]{%
    \errmessage{(Inkscape) Color is used for the text in Inkscape, but the package 'color.sty' is not loaded}%
    \renewcommand\color[2][]{}%
  }%
  \providecommand\transparent[1]{%
    \errmessage{(Inkscape) Transparency is used (non-zero) for the text in Inkscape, but the package 'transparent.sty' is not loaded}%
    \renewcommand\transparent[1]{}%
  }%
  \providecommand\rotatebox[2]{#2}%
  \ifx\svgwidth\undefined%
    \setlength{\unitlength}{219.32105616bp}%
    \ifx\svgscale\undefined%
      \relax%
    \else%
      \setlength{\unitlength}{\unitlength * \real{\svgscale}}%
    \fi%
  \else%
    \setlength{\unitlength}{\svgwidth}%
  \fi%
  \global\let\svgwidth\undefined%
  \global\let\svgscale\undefined%
  \makeatother%
  \begin{picture}(1,0.47081137)%
    \put(0,0){\includegraphics[width=\unitlength]{fig-10.pdf}}%
    \put(0.28870413,0.03320326){\color[rgb]{0,0,0}\makebox(0,0)[lb]{\smash{$\lambda_1$}}}%
    \put(0.12091351,0.18952988){\color[rgb]{0,0,0}\makebox(0,0)[lb]{\smash{$\mu_1$}}}%
    \put(0.45623418,0.27548704){\color[rgb]{0,0,0}\makebox(0,0)[lb]{\smash{${\small(\lambda^{*},\mu^{*})}$}}}%
    \put(0.62975668,0.43025026){\color[rgb]{0,0,0}\makebox(0,0)[lb]{\smash{$\gamma^{*}_{+}$}}}%
    \put(0.32882794,0.21868838){\color[rgb]{0,0,0}\makebox(0,0)[lb]{\smash{$\gamma^{*}_{-}$}}}%
    \put(0.32644804,0.35475165){\color[rgb]{0,0,0}\makebox(0,0)[lb]{\smash{{\color{blue}$\gamma^{*}$}}}}%
    \put(0.6912121,0.03267329){\color[rgb]{0,0,0}\makebox(0,0)[lb]{\smash{$\lambda$}}}%
    \put(0.11855277,0.40876347){\color[rgb]{0,0,0}\makebox(0,0)[lb]{\smash{$\mu$}}}%
    \put(0.38763512,0.23817191){\color[rgb]{0,0,0}\makebox(0,0)[lb]{\smash{$\sigma$}}}%
    \put(0.20977649,0.11219469){\color[rgb]{0,0,0}\makebox(0,0)[lb]{\smash{$\sigma$}}}%
  \end{picture}%
\endgroup%

%% file: fig-11.pdf_tex
%% Creator: Inkscape inkscape 0.48.5, www.inkscape.org
%% Accompanies image file 'fig-11.pdf' (pdf, eps, ps)
\begingroup%
  \makeatletter%
  \providecommand\color[2][]{%
    \errmessage{(Inkscape) Color is used for the text in Inkscape, but the package 'color.sty' is not loaded}%
    \renewcommand\color[2][]{}%
  }%
  \providecommand\transparent[1]{%
    \errmessage{(Inkscape) Transparency is used (non-zero) for the text in Inkscape, but the package 'transparent.sty' is not loaded}%
    \renewcommand\transparent[1]{}%
  }%
  \providecommand\rotatebox[2]{#2}%
  \ifx\svgwidth\undefined%
    \setlength{\unitlength}{211.50911354bp}%
    \ifx\svgscale\undefined%
      \relax%
    \else%
      \setlength{\unitlength}{\unitlength * \real{\svgscale}}%
    \fi%
  \else%
    \setlength{\unitlength}{\svgwidth}%
  \fi%
  \global\let\svgwidth\undefined%
  \global\let\svgscale\undefined%
  \makeatother%
  \begin{picture}(1,0.46755749)%
    \put(0,0){\includegraphics[width=\unitlength]{fig-11.pdf}}%
    \put(0.30693188,0.03446404){\color[rgb]{0,0,0}\makebox(0,0)[lb]{\smash{$\lambda_1$}}}%
    \put(0.12537936,0.18899981){\color[rgb]{0,0,0}\makebox(0,0)[lb]{\smash{$\mu_1$}}}%
    \put(0.66584926,0.41598449){\color[rgb]{0,0,0}\makebox(0,0)[lb]{\smash{$\gamma^{*}_{+}$}}}%
    \put(0.34097297,0.21923526){\color[rgb]{0,0,0}\makebox(0,0)[lb]{\smash{$\gamma^{*}_{-}$}}}%
    \put(0.33985588,0.31257198){\color[rgb]{0,0,0}\makebox(0,0)[lb]{\smash{{\color{blue}$\gamma^{*}$}}}}%
    \put(0.4743967,0.27899145){\color[rgb]{0,0,0}\makebox(0,0)[lb]{\smash{$\bar{\sigma}$}}}%
    \put(0.1262577,0.41442363){\color[rgb]{0,0,0}\makebox(0,0)[lb]{\smash{$\mu$}}}%
    \put(0.72656577,0.03541932){\color[rgb]{0,0,0}\makebox(0,0)[lb]{\smash{$\lambda$}}}%
  \end{picture}%
\endgroup%